\documentclass{article}
\usepackage{amsfonts}
\usepackage{amssymb}
\usepackage{amsmath}
\usepackage{graphicx}
\usepackage{xy} \xyoption{all}
\newtheorem{theorem}{Theorem}[section]
\newtheorem{acknowledgement}[theorem]{Acknowledgement}

\newtheorem{corollary}[theorem]{Corollary}

\newtheorem{definition}[theorem]{Definition}
\newtheorem{example}[theorem]{Example}

\newtheorem{lemma}[theorem]{Lemma}

\newtheorem{proposition}[theorem]{Proposition}
\newtheorem{remark}[theorem]{Remark}

\newenvironment{proof}[1][Proof]{\noindent\textbf{#1.} }{\ \rule{0.5em}{0.5em}}
\begin{document}

\title{Leavitt path algebras with at most countably many irreducible representations}
\author{Pere Ara\\Department de Matematiques\\Universitat Autonoma de Barcelona\\08913 Bellaterra (Barcelona), Spain.\\E-mail: para@mat.uab.cat
\and Kulumani M. Rangaswamy\\Department of Mathemastics\\University of Colorado at Colorado Springs\\Colorado Springs, Colorado 80918, USA.\\E-mail: krangasw@uccs.edu }
\date{}
\maketitle

\begin{abstract}
Let $E$ be an arbitrary directed graph with no restrictions on the number of
vertices and edges and let $K$ be any field. We give necessary and sufficient
conditions for the Leavitt path algebra $L_{K}(E)$ to be of countable
irreducible representation type, that is, we determine when $L_{K}(E)$ has at most
countably many distinct isomorphism classes of simple left $L_{K}(E)$-modules.
\ It is also shown that $L_{K}(E)$ has finitely many isomorphism classes of
simple left modules if and only if $L_{K}(E)$ is a semi-artinian von Neumann
regular ring with finitely many ideals. Equivalent conditions on the
graph $E$ are also given. Examples are constructed showing that for each
(finite or infinite) cardinal $\kappa$ there exists a Leavitt path algebra
$L_{K}(E)$ having exactly $\kappa$ distinct isomorphism classes of simple
right modules.

\end{abstract}

\section{Introduction}

The notion of Leavitt path algebras was introduced and initially studied in
\cite{AA}, \cite{AMP} as algebraic analogues of graph C$^{\ast}$-algebras and
the study of their various ring-theoretic properties has been the subject of a
series of papers in recent years (see, e.g., \cite{AA} - \cite{AMMS},
\cite{R-1}, \cite{T}). In \cite{GR}, Goncalves and Royer indicated a method of
constructing various representations of a Leavitt path algebra $L_{K}(E)$ over
a graph $E$ by using the concept of algebraic branching systems. Expanding
this, Chen \cite{C} studied special types of irreducible representations of
$L_{K}(E)$ by using the sinks as well as the infinite paths which are\ not
tail-equivalent in the graph $E$ and he noted that these can also be
considered as algebraic branching systems. For additional ways of constructing
irreducible representations of $L_{K}(E)$ see \cite{ARa}.

In this paper we investigate the Leavitt path algebras $L_{K}(E)$ which are of
countable irreducible representation type, that is, $L_{K}(E)$ having at most
countably many distinct isomorphism classes of simple left/right $L_{K}%
(E)$-modules. For short, we call a Leavitt path algebra with this property as
CIRT. In many of the past investigations describing various algebraic
properties of a Leavitt path algebra $L_{K}(E)$, the individual nature of the
field $K$ did not seem to play any role at all. In this context, it may be of
of some interest to note that, when $L_{K}(E)$ is CIRT and the graph $E$
contains cycles, $K$ must necessarily be a countable field. Our structure
theorem states that $L_{K}(E)$ is CIRT if and only if it is the union of a
smooth ascending chain of countable length consisting of graded ideals
\[
0<I_{1}<\cdot\cdot\cdot<I_{\alpha}<I_{\alpha+1}<\cdot\cdot\cdot\text{\qquad
\qquad\qquad(}\alpha<\tau\text{)}\qquad\qquad(\ast\ast)
\]
where $\tau$ is a countable ordinal and, for each $0\leq\alpha<\tau$,
$I_{\alpha+1}/I_{\alpha}$ is a direct sum of at most countably many matrix
rings over $K$ and/or $K[x,x^{-1}]$. Moreover, the field $K$ is countable
whenever the graph $E$ contains cycles (equivalently, $K[x,x^{-1}]$ occurs as
a direct factor in $I_{\alpha+1}/I_{\alpha}$ for some $\alpha$).

We also show that $L_{K}(E)$ will have at most finitely many non-isomorphic
simple left/right $L_{K}(E)$-modules if and only if $L_{K}(E)$ is a
semi-artinian von Neumann regular ring with at most finitely many ideals. In
particular, when $E$ is a finite graph, then $L_{K}(E)$ has this property
exactly when $L_{K}(E)$ is an artinian semisimple ring (equivalently, the
graph $E$ is acyclic). We also construct, for each arbitrary (finite or
infinite) cardinal $\kappa$, a Leavitt path algebra for which the cardinality
of the distinct isomorphism classes of simple right modules is exactly
$\kappa$.

\section{Preliminaries}

For the general notation, terminology and results on Leavitt path algebras, we
refer the reader to \cite{AA}, \cite{AAS} and \cite{T}. We just give a short
outline of some of the needed concepts. A (directed) graph $E=(E^{0}%
,E^{1},r,s)$ consists of two sets $E^{0}$ and $E^{1}$ together with maps
$r,s:E^{1}\rightarrow E^{0}$. All the graphs $E$ that we consider here are
arbitrary with no restrictions on the number of vertices and the number of
edges emitted by a vertex. Also $K$ stands for an arbitrary field. For each
$e\in E^{1}$, we call $e^{\ast}$ a \textit{ghost} edge. We let $r(e^{\ast})$
denote $s(e)$, and we let $s(e^{\ast})$ denote $r(e)$. The set of all vertices
on the path $\mu$ is denoted by $\mu^{0}$. A vertex $v$ in $E$ is said to be {\it regular}
if $0<|s^{-1}(v)|<\infty $. A {\it singular} vertex is a vertex which is not regular.

Given an arbitrary graph $E$ and a field $K$, the \textit{Leavitt path }%
$K$\textit{-algebra }$L_{K}(E)$ is defined to be the $K$-algebra generated by
a set $\{v:v\in E^{0}\}$ of pairwise orthogonal idempotents together with a
set of variables $\{e,e^{\ast}:e\in E^{1}\}$ which satisfy the following conditions:

(1) \ $s(e)e=e=er(e)$ for all $e\in E^{1}$;

(2) $r(e)e^{\ast}=e^{\ast}=e^{\ast}s(e)$\ for all $e\in E^{1}$;

(3) (The ``CK-1 relations") For all $e,f\in E^{1}$, $e^{\ast}e=r(e)$ and
$e^{\ast}f=0$ if $e\neq f$.

(4) (The ``CK-2 relations") For every regular vertex $v\in E^{0}$,
\[
v=\sum_{e\in E^{1},s(e)=v}ee^{\ast}.
\]

A path $\mu$ $=e_{1}\dots e_{n}$ in $E$ is \textit{closed} if $r(e_{n}%
)=s(e_{1})$, in which case $\mu$ is said to be based at the vertex $s(e_{1})$.
The closed path $\mu$ is called a \textit{cycle} if $s(e_{i})\neq s(e_{j})$
for every $i\neq j$. A graph $E$ is said to satisfy \textit{Condition }$(K)$
provided no vertex $v\in E^{0}$ is the base of precisely one simple closed
path, i.e., either no simple closed path is based at $v$, or at least two are
based at $v$. An \textit{exit }for a path $\mu=e_{1}\dots e_{n}$ is an edge
$e$ such that $s(e)=s(e_{i})$ for some $i$ and $e\neq e_{i}$.

If there is a path from vertex $u$ to a vertex $v$, we write $u\geq v$. A
subset $D$ of vertices is said to be \textit{downward directed }\ if for any
$u,v\in D$, there exists a $w\in D$ such that $u\geq w$ and $v\geq w$. A
subset $H$ of $E^{0}$ is called \textit{hereditary} if, whenever $v\in H$ and
$w\in E^{0}$ satisfy $v\geq w$, then $w\in H$. A hereditary set is
\textit{saturated} if, for any regular vertex $v$, $r(s^{-1}(v))\subseteq H$
implies $v\in H$.

For any vertex $v$, the \textit{tree} of $v$ is $T(v)=\{w:v\geq w\}$. We say
there is a bifurcation at a vertex $v$, if $v$ emits more than one edge. In a
graph $E$, a vertex $v$ is called a \textit{line point} if there is no
bifurcation or a cycle based at any vertex in $T(v)$. Thus, if $v$ is a line
point, there will be a single finite or infinite line segment $\mu$ starting
at $v$ ($\mu$ could just be $v$) and any other path $\alpha$ with
$s(\alpha)=v$ will just be an initial sub-segment of $\mu$. It was shown in
\cite{AMMS} that $v$ is a line point in $E$ if and only if $vL_{K}(E)$ (and
likewise $L_{K}(E)v$) is a simple right (left) ideal. Moreover, the ideal
generated by all the line points in $E$ is the socle of $L_{K}(E)$. If $v$ is
a line point, then it is clear that any $w\in T(v)$ is also a line point.

We shall be using the following concepts and results from \cite{T} in our
investigation. Although it is assumed in \cite{T} that all the graphs are
countable, i.e. have at most countably many vertices and edges, an examination
of the proofs shows that this assumption is not needed at all. A
\textit{breaking vertex }of a hereditary saturated subset $H$ is an infinite
emitter $w\in E^{0}\backslash H$ (that is, $s^{-1}(w)$ is an infinite set)
with the property that $1\leq|s^{-1}(w)\cap r^{-1}(E^{0}\backslash H)|<\infty
$. The set of all breaking vertices of $H$ is denoted by $B_{H}$. For any
$v\in B_{H}$, $v^{H}$ denotes the element $v-\sum_{s(e)=v,r(e)\notin
H}ee^{\ast}$. Given a hereditary saturated subset $H$ and a subset $S\subseteq
B_{H}$, $(H,S)$ is called an \textit{admissible pair.} Given an admissible
pair $(H,S)$, $I_{(H,S)}$ denotes the ideal generated by $H\cup\{v^{H}:v\in
S\}$. It was shown in \cite{T} that the graded ideals of $L_{K}(E)$ are
precisely the ideals of the form $I_{(H,S)}$ for some admissibile pair $(H,S)$.
Moreover, $L_{K}(E)/I_{(H,S)}\cong L_{K}(E\backslash(H,S))$. Here
$E\backslash(H,S)$ is the \textit{Quotient graph of }$E$ in which\textit{\ }%
$(E\backslash(H,S))^{0}=(E^{0}\backslash H)\cup\{v^{\prime}:v\in
B_{H}\backslash S\}$ and $(E\backslash(H,S))^{1}=\{e\in E^{1}:r(e)\notin
H\}\cup\{e^{\prime}:e\in E^{1},r(e)\in B_{H}\backslash S\}$ and $r,s$ are
extended to $(E\backslash(H,S))^{0}$ by setting $s(e^{\prime})=s(e)$ and
$r(e^{\prime})=r(e)^{\prime}$.

If $p=e_{1}e_{2}\cdot\cdot\cdot e_{n}\cdot\cdot\cdot$ is an infinite path
where the $e_{i}$ are edges, then for any positive integer $n$, let
$\tau_{\leq n}(p)=e_{1}e_{2}\cdot\cdot\cdot e_{n}$ and $\tau_{>n}%
(p)=e_{n+1}e_{n+2}\cdot\cdot\cdot$. Two infinite paths $p$ and $q$ are said to
be \textit{tail} \textit{equivalent}, in symbols, $p\sim q$, if there exist
positive integers $m$ and $n$ such that $\tau_{>n}(p)=\tau_{>m}(q)$. Then
$\sim$ is an equivalence relation. Given an equivalence class of infinite
paths $[p]$, Chen \cite{C} defines an $L_{K}(E)$-module action on
the linear span $V_{[p]}$ of the equivalence class $[p]$ of $p$ and shows that $V_{[p]}$ becomes a
simple left $L_{K}(E)$-module. He further shows that, for infinite paths
$p,q$, $V_{[p]}\cong V_{[q]}$ if and only if $p\sim q$.

\section{Leavitt path algebras of countable irreducible representation type}

Throughout this and the following sections, $L$ will denote the Leavitt path
algebra $L_{K}(E)$. In this section, we give a complete description of the
Leavitt path algebra $L$ which has at most countably many distinct isomorphism
classes of simple left/right $L$-modules. For short, we call a Leavitt path
algebra with this property as CIRT. 

The proof of our main theorem is derived from the following series of
Propositions containing the necessary conditions for $L$ to be CIRT.

\begin{proposition}
\label{disjoint cycles}If $L$ is CIRT, then distinct cycles in $E$ must be disjoint.
\end{proposition}

\begin{proof}
We may assume that $E$ contains cycles, since there is nothing to prove if $E$
is acyclic. Suppose that there are two different cycles $g,h$ based at the
same vertex $v$. Consider the infinite path $ggg\cdot\cdot\cdot$ which, for
convenience, we write as $p=g_{1}g_{2}g_{3}\cdot\cdot\cdot$ indexed by the set
$\mathbf{P}$ of positive integers where $g_{i}=g$ for all $i$. Now for every
subset $S$ of $\mathbf{P}$, define an infinite path $p_{S}$ by replacing
$g_{i}$ by $h$ if and only if $i\in S$. Observe that this gives an uncountable
family of pairwise different infinite paths on $E$. Denote by $Y$ this family
of infinite paths. Now, if $q=\beta_{1}\beta_{2}\cdots$ is a path in $Y$, then
the paths in $Y$ which are tail-equivalent to $q$ have the form $\gamma
_{1}\cdots\gamma_{j}\beta_{i+1}\beta_{i+2}\beta_{i+3}\cdots$, where $i,j$ are
non-negative integers and $\gamma_{1},\dots,\gamma_{j}\in\{g,h\}$. This is a
countable family and so each tail-equivalence class contains countably many
infinite paths in $Y$. Consequently, there are uncountably many distinct tail
equivalence classes of paths in $Y$. By Chen \cite{C}, $L$ then has
uncountably many non-isomorphic simple left/right $L$-modules. This
contradicts the fact that $L$ is a CIRT. Hence no two cycles in $E$ have a
common vertex.
\end{proof}

\begin{proposition}
\label{countable field}If the graph $E$ contains cycles and $L$ is CIRT, then
the field $K$ must be countable.
\end{proposition}

\begin{proof}
Consider a cycle $c$ based at a vertex $v$ in $E$. Then $H=\{u\in E^{0}:u\ngeqq
v\}$ is a hereditary saturated set and $E^{0}\backslash H=$ $M(v)$ is downward
directed. By Proposition \ref{disjoint cycles}, distinct cycles have no common
vertex. Hence $c$ is a cycle without exits in $M(v)$. By Theorems 3.12 and 4.3
of \cite{R-1}, each irreducible polynomial $g(x)\in K[x,x^{-1}]$ then gives
rise to a primitive ideal $P_{g}$ of $L$, generated by $I_{(H,B_{H})}\cup\{g(c)\}$ 
(and a corresponding simple module having $P_{g}$ as its annihilator). Since
$L$ is CIRT, $K[x,x^{-1}]$ cannot have uncountably many irreducible
polynomials. This implies that $K$ must be a countable field.
\end{proof}

\bigskip

 We next establish an important property for $E$ when $L$ is CIRT.

\begin{proposition}
\label{CIRT implies line points or NE cycles}If $L$ is CIRT, then the graph
$E$ contains line points or cycles without exits, or both.
\end{proposition}

\begin{proof}
Suppose, by way of contradiction, the graph $E$ contains no line points and no
cycles without exits. \ 

Let $v$ be a vertex in $E$, and set $F=T(v)$, seen as a subgraph of $E$, that
is, we consider the tree of $v$ together with all edges in $E$ connecting the
vertices of $T(v)$. It is enough to show that $F$ has uncountably many
tail-equivalence classes of infinite paths. We wish to build a countable
complete subgraph $G$ of $F$ having the same properties.

First, we construct an increasing family of countable complete subgraphs
$F_{n}$ starting with $F_{0}=\{v\}$. Indeed, the graphs $F_n\subseteq F$ will satisfy
the property that, for $u\in(F_{n})^{0}$,
if $u$ is a regular vertex in $F_{n}$, then $s_{F_{n}}^{-1}(u)=s_{E}^{-1}(u)$.
This clearly implies that $F_n$ is a complete subgraph of $E$,
see \cite[Definition 1.6.7]{AAS}. Now assume that $F_{n}$ has been built for
some $n\geq0$. For each vertex $u\in(F_{n})^{0}$ which is a sink in
$F_{n}$ and which is a regular vertex in $E$, add to $F_{n}$ all the
finitely many edges emitted by $u$ in $E$ and the end vertices of these edges.
For each vertex $u\in(F_{n})^{0}$ which is a sink in $F_{n}$ and which
is an infinite emitter in $E$, add a countable number of edges emitted by $u$
in $E$ and the corresponding end vertices of these edges. This way we get the
new subgraph $F_{n+1}$. Let $G$ be the union of all the $F_{n}$'s. Then $G$ is
a complete subgraph of $F$, and $G$ contains no line points and no cycles
without exits.

Clearly $L_{K}(G)$ is a subalgebra of $L_{K}(F)$ and the map sending an
infinite path in $G$ to the same infinite path in $F$ induces an injective map
from the set $T_{G}$ of the tail-equivalence classes of infinite paths in $G$
to the set $T_{F}$ of the corresponding tail-equivalence classes of infinite
paths in $F$. Thus if $T_{G}$ is uncountable so is $T_{F}$.

To show that the set $T_{G}$ is uncountable, let $X$ be the set of all
infinite paths in $G$ together with the finite paths in $G$ ending in a
singular vertex. For each finite path $\gamma$ (which could also be a vertex),
define 
$$\mathcal{Z}(\gamma)=\{  p\in X \text{ such that } p=\gamma p^{\prime},
\text{ where } p^{\prime} \text{ is some path in } G \}.$$ Then $X$ can be made a locally compact
Hausdorff space, and the collection
\[
\{\mathcal{Z}(\gamma)\setminus\bigcup_{e\in\mathcal{F}}\mathcal{Z}(\gamma
e)\},
\]
for $\gamma$ a finite path and $\mathcal{F}$ a finite set of edges starting at
$r(\gamma)$, is a basis of compact open sets in $X$ (see Theorem 2.1,
\textbf{\cite{W}}). Since $G$ is a countable graph, these sets form a
countable basis of open and closed (compact) sets for $X$ and hence $X$ is
second countable. Since $G$ does not have line points or cycles without exits,
we see that $X$ has no isolated points.

Now $\mathcal{Z}(v)$ is a compact Hausdorff totally disconnected space with no
isolated points and hence is isomorphic to the Cantor space (see Theorem 13,
\cite{MHS}). This means $\mathcal{Z}(v)$ contains $2^{\aleph_{0}}$ infinite
paths. Since $G$ is countable, an argument similar to the one used in the
proof of Proposition \ref{disjoint cycles} shows that the tail-equivalence
classes of paths in $\mathcal{Z}(v)$ are at most countable. Since
$|\mathcal{Z}(v)|=2^{\aleph_{0}}$, we conclude that $G$ (and hence $E$)
contains $2^{\aleph_{0}}$ infinite paths which are not mutually
tail-equivalent. By Chen \cite{C}, $L$ will then have $2^{\aleph_{0}}$
non-isomorphic simple left/right $L$-modules, a contradiction.

Hence $E$ always contains line points and/or cycles without exits.
\end{proof}

 For a ring $R$ and any index set 
 $\Lambda$, we denote by $M_{\Lambda}(R)$ the ring of matrices with coefficients in $R$, indexed on $\Lambda$, 
 and having only finitely many nonzero entries. 

We are now ready to prove the main result of this section.
 
\begin{theorem}
\label{CIRT}Let $E$ be an arbitrary graph and $K$ be any field. Then the
following properties are equivalent:

(i) $\ \ L=L_{K}(E)$ is CIRT;

(ii) \ $L$ is the union of a well-ordered smooth ascending chain of countable
length consisting of graded ideals
\[
0<I_{1}<\cdot\cdot\cdot<I_{\alpha}<I_{\alpha+1}<\cdot\cdot\cdot\text{\qquad
\qquad\qquad(}\alpha<\tau\text{)}\qquad\qquad(\ast\ast)
\]
where, $\tau$ is a countable ordinal, for each $0\leq\alpha<\tau$,
$I_{\alpha+1}/I_{\alpha}$ is a direct sum of at most countably many matrix
rings $M_{\Lambda}(R)$, where $\Lambda$ are arbitrarily-sized index sets and $R$ is either $K$ or $K[x,x^{-1}]$. 
Moreover, $K$ will be a countable field
whenever $E$ contains cycles.
\end{theorem}

\begin{proof}
Assume (i). By Proposition \ref{CIRT implies line points or NE cycles}, the
graph $E$ contains line points and/or cycles without exits. Let $I_{1}$ be the
ideal generated by all the line points and vertices on all the cycles without
exits in $E$. Then $I_{1}=I(H,\emptyset)$, where $H=I_{1}\cap E^{0}$. 
By \cite{AAS}, $I_{1}$ is a direct sum of matrices of the form
$M_{n_{i}}(K)$ and/or $M_{l_{j}}(K[x,x^{-1}])$ where $n_{i},l_{j}$ are cardinal numbers. 
Suppose we have defined a graded ideal $I_{\alpha
}=I(H_{\alpha},B_{\alpha})$ for some $\alpha\geq1$ where $H_{\alpha}
=I_{\alpha}\cap E^{0}$ and $B_{\alpha}=\{v\in B_{H_{\alpha}}:v^{H_{\alpha}}\in
I_{\alpha}\}$. If $I_{\alpha}\neq L$, consider $L/I_{\alpha}\cong
L_{K}(E\backslash (H_{\alpha},B_{\alpha}))$. Now $L/I_{\alpha}$ is CIRT and
hence, by Proposition \ref{CIRT implies line points or NE cycles},
$E\backslash(H_{\alpha},B_{\alpha})$ contains line points and/or cycles
without exits. Define the ideal $I_{\alpha+1}\supset I_{\alpha}$ so that
$I_{\alpha+1}/I_{\alpha}$ is the ideal generated by all the line points and
the vertices on all the cycles without exits in $E\backslash(H_{\alpha
},B_{\alpha})$. By \cite{AAS}, $I_{\alpha+1}/I_{\alpha}$ is then a direct sum
of matrices of the form $M_{n_{i}}(K)$ and/or $M_{l_{j}}(K[x,x^{-1}])$ where
$n_{i},l_{j}$ are cardinal numbers. If $\gamma$ is a limit ordinal and $I_{\alpha}$ has been
defined for all $\alpha<\gamma$, then define $I_{\gamma}=
{\textstyle\bigcup\limits_{\alpha<\gamma}}
I_{\alpha}$. By transfinite induction, we then obtain a smooth ascending chain
$(\ast\ast)$ of graded ideals with the desired properties and its union is
$L$. Now each graded ideal is isomorphic to the Leavitt path algebra of a
suitable graph (see \cite{RT}) and so is a ring with local units. Hence every
left/right ideal of $I_{\alpha+1}$ is also a left/right ideal of $L$. It is
then readily seen that every simple $I_{\alpha+1}$-module is isomorphic to a
simple $L$-module. Since $L$ is CIRT, the length of the chain $\tau$ is
countable and that, for each $\alpha<\tau$, $I_{\alpha+1}/I_{\alpha}$ is a
direct sum of at most countably many matrices of the form $M_{n_{i}}(K)$
and/or $M_{l_{j}}(K[x,x^{-1}])$.

Moreover, Proposition \ref{countable field} shows that the field $K$ must be
countable if $E$ contains cycles. This proves (ii).

Assume (ii) so that $L$ is the union of a chain of graded ideals 
satisfying the stated properties. Let $S=L/M$ be a simple module where $M$ is a
maximal left ideal of $L$. Let $\beta$ be the smallest ordinal such that
$I_{\beta}\nsubseteq M$. Clearly $\beta$ is a non-limit ordinal. Let
$\beta=\alpha+1$. Then $I_{\alpha}\subseteq M$. So
\[
S=(M+I_{\alpha+1})/M\cong I_{\alpha+1}/(I_{\alpha+1}\cap M)\cong(I_{\alpha
+1}/I_{\alpha})/[(I_{\alpha+1}\cap M)/I_{\alpha}]\text{.}%
\]
Thus every simple left $L$-module is isomorphic to a simple left module over
the ring $I_{\alpha+1}/I_{\alpha}$ for some $\alpha$ with $0\leq\alpha<\tau$.
Now for each $\alpha$, $I_{\alpha+1}/I_{\alpha}$ is CIRT, due to the fact that
$I_{\alpha+1}/I_{\alpha}$ is a direct sum of at most countably many matrix
rings $M_{n_{k}}(K)$ and matrix rings $M_{r_{j}}(K[x,x^{-1}])$, where
$n_{k}, r_{j}$ are arbitrarily-sized cardinal numbers  and that $K$ is countable whenever $r_{j}\neq 0$ for some
$r_{j}$. Since $\tau$ is a countable ordinal, it is readily seen that $L$ is
also CIRT.
\end{proof}

\begin{remark}
In Theorem \ref{CIRT}, if the graph $E$ is acyclic, then for every
$\alpha<\tau$, $I_{\alpha+1}/I_{\alpha}$ is isomorphic to a direct sum of
matrix rings over $K$ \ and hence is a direct sum of simple modules. Thus the
chain $(\ast\ast)$ becomes the socular chain for $L$. In other words, $L$
becomes a semi-artinian von Neumann regular ring with countable Loewy length.
Note that, in this case, there is no restriction on the cardinality of the
field $K$.
\end{remark}

\section{Leavitt path algebras of finite irreducible representation type}

We now wish to specialize to the case when the Leavitt path algebra
$L=L_{K}(E)$ is of finite irreducible representation type, that is, when $L$
has at most finitely many distinct isomorphism classes of simple left
$L$-modules. In this case, $L$ satisfies conditions stronger than those of
Theorem \ref{CIRT}. Specifically, $L$ becomes a semi-artinian von Neumann
regular ring with finitely many two-sided ideals. Moreover, the result holds
for arbitrary fields $K$ without any restrictions.

We begin with the following Lemma and its Corollary which are used in the
proof of Theorem \ref{FIRT}.

\begin{lemma}
\label{Quotient simple}Suppose $R$ is a von Neumann regular ring and $A$ is a non-zero proper ideal of $R$. 
Then every simple left $R$-module is
isomorphic to a simple left module over $R/A$ or $A$. Conversely, every
simple left module over $R/A$, or $A$ is isomorphic to some simple
left $R$-module.
\end{lemma}

\begin{proof}
Let $R/M$ be a simple left $R$-module, where $M$ is a maximal left ideal of
$R$. Then it is easy to see that $R/M$ is isomorphic to a simple left module
over $R/A$ or $A$ according as (i) $M\supset A$, or (ii) $M\nsupseteq A$.

To prove the converse, we only consider the case of a simple left $A$-module
$S$, as the other case is trivial. Write $S=A/N$, where $N$ is a maximal left ideal of $A$.
For $n\in N$ there is an idempotent $e$ in $A$ such that $en=n$. It follows from this that 
$N$ is indeed a left $R$-module, and so $S=A/N$ is a simple $R$-module. 
\end{proof}

\begin{corollary}
\label{Semi-artinian}Suppose $L_{K}(E)$ is semi-artinian ring with finitely
many ideals. Then $L_{K}(E)$ has at most finitely many non-isomorphic simple
left modules.
\end{corollary}

\begin{proof}
Suppose the semi-artinian ring $L=L_{K}(E)$ has only finitely many two-sided
ideals. Then we can build a finite ascending chain of two-sided ideals
$0<S_{1}<\cdot\cdot\cdot<S_{n}=L$ where, for each $i=1,\cdot\cdot\cdot,n-1$,
$S_{i+1}/S_{i}$ is a simple ring which, being also semi-artinian, is a direct
sum of isomorphic simple left ideals. Consequently, each $S_{i+1}/S_{i}$ has a
single isomorphism class of simple left $S_{i+1}/S_{i}$-modules. Also by
\cite{ARS}, $L_{K}(E)$ is von Neumann regular. We then conclude, by Lemma
\ref{Quotient simple}, that $L$ has exactly $n$ distinct isomorphism classes
of simple left $L$-modules.
\end{proof}

\begin{theorem}
\label{FIRT}Let $E$ be an arbitrary graph and $K$ be any field. Then the
following are equivalent for the Leavitt path algebra $L=L_{K}(E)$:

(i) \ $L$ has at most finitely many non-isomorphic simple left/right $L$-modules;

(ii) $L$ is a semi-artinian ring with finitely many two-sided ideals.
\end{theorem}

\begin{proof}
Assume (i). Since there are only finitely many non-isomorphic simple left
modules, $L_{K}(E)$ has only finitely many distinct primitive ideals. We wish
to show that there are no non-graded prime ideals in $L_{K}(E)$. Now any
non-graded prime ideal of $L_{K}(E)$ is primitive (see Theorem 4.3,
\cite{R-1}). From Theorem 3.12 of \cite{R-1}, the non-graded prime ideals of
$L_{K}(E)$ are precisely the ideals $P$ generated by $I_{(H,S)}\cup\{f(c)\}$, where 
$P\cap E^{0}=H$, $S=\{v\in B_{H}:v^{H}\in P\}$, $c$ is a
cycle without exits in $E\backslash H$, and $f$ \ is an irreducible polynomial
in $K[x,x^{-1}]$. So, if $L_{K}(E)$ has one such non-graded prime ideal $P$,
then corresponding to each of the infinitely many irreducible polynomials $g$
in $K[x,x^{-1}]$, $L_{K}(E)$ will have a non-graded prime (and hence a
primitive) ideal $P_{g}$ generated by $I_{(H,S)}\cup\{g(c)\}$. This
contradicts that $L_{K}(E)$ has only a finite number of primitive ideals. Thus
every prime ideal of $L_{K}(E)$ must be graded and we appeal to Corollary 3.13
of \cite{R-1} to conclude that the graph $E$ satisfies Condition (K). From
Theorem 6.16 of \cite{T}, we then conclude that every ideal of $L_{K}(E)$ is graded.

Thus if $I$ is an ideal of $L_{K}(E)$ then, being a graded ideal,
$I=I_{(H,S)}$, where $H=I\cap E^{0}$ and $S=\{v\in B_{H}:v^{H}\in I\}$ and,
moreover, $L_{K}(E)/I\cong L_{K}(E\backslash(H,S))$ by \cite{T}. As the
Jacobson radical of the Leavitt path algebra $L_{K}(E\backslash(H,S))$ is zero
(see \cite{ARS}), we conclude that $I$ is the intersection of all the
primitive ideals containing $I$. As there are only finitely many primitive
ideals in $L_{K}(E)$, we then conclude that $L_{K}(E)$ contains only finitely
many distinct ideals which are all of the form $I_{(H,S)}$ for some admissible
pair $(H,S)$.

We also claim that $E$ contains no cycles. Because if there is a cycle $g$ in
$E$, then by Condition (K) there will be another cycle $h\neq g$ sharing a
common vertex with $g$ and this contradicts Proposition \ref{disjoint cycles}.
Hence $E$ is acyclic and $L$ is von Neumann regular, by \cite{AR1}. Also, by
Proposition \ref{CIRT implies line points or NE cycles}, $E$ contains line
points and so $S_{1}=Soc(L)\neq0$. Also, since there are only finitely many
ideals, $S_{1}$ is a direct sum of finitely many homogeneous components:
$S_{1}=%
{\textstyle\bigoplus\limits_{i=1}^{n_{1}}}
$ $S_{1i}$ with $S_{1i}$ a direct sum of isomorphic simple left ideals of $L$.
Let $H=S_{1}\cap E^{0}$. Now $L/S_{1}\cong L_{K}(E\backslash(H,\emptyset ))$
satisfies the same hypothesis as $L$ and hence has a non-zero socle
$S_{2}/S_{1}$. Proceeding like this and using the fact that $L$ has at most
finitely many ideals, we conclude $L$ is the union of a finite ascending chain
of ideals $\{0\}=S_{0}\subset S_{1}\subset...\subset S_{m}=L$ where, for each
$i$, $S_{i+1}/S_{i}=Soc(L/S_{i})$ and is a direct sum of finitely many
homogeneous components. This proves (ii).

(ii)$\Longrightarrow$(i). This follows from Corollary \ref{Semi-artinian}.
\end{proof}

\begin{remark}
One might wonder what are the finite lattices appearing as ideal lattices of
the algebras $L_{K}(E)$ characterized in Theorem \ref{FIRT}. Indeed, observe
that the lattice of ideals of a von Neumann regular ring is a distributive
lattice. It follows from Bergman's Theorem \cite{Bergman} that every finite
distributive lattice can be represented as the lattice of ideals of a unital
ultramatricial algebra. On the other hand, \cite[Proposition 2.12]{Raeburn}
asserts that this unital ultramatricial algebra is Morita-equivalent to a
Leavitt path algebra $L_{K}(E)$ of an acyclic graph. For this graph $E$, the
conditions in Theorem \ref{FIRT}(ii) necessarily hold. In conclusion, we see
that the class of lattices appearing as ideal lattices of the algebras
described in Theorem \ref{FIRT} is exactly the class of all finite
distributive lattices.
\end{remark}

\section{Equivalent Graphical Conditions}

In this section we describe the graphical properties of $E$ under which the
Leavitt path algebra $\ L_{K}(E)$ is of finite irreducible representation
type. We begin with a simple lemma describing when two line points generate
isomorphic simple right ideals.

\begin{lemma}
\label{Line point} Given two line points $u,v$, $uL_{K}(E)\cong vL_{K}(E)$ if
and only if $T(u)\cap T(v)$ is not empty.
\end{lemma}

\begin{proof}
Suppose $\theta:u$ $L_{K}(E)\rightarrow v$ $L_{K}(E)$ is an isomorphism. Then
$\theta$ is given by the left multiplication by the non-zero element
$\theta(u)=vsu$ for some $s\in L_{K}(E)$. We can clearly assume that $s=vsu$.
Write $s=%
{\displaystyle\sum\limits_{i=1}^{m}}
k_{i}\alpha_{i}\beta_{i}^{\ast}$ where $k_{i}\in K$ and $\alpha_{i},\beta_{i}$
are finite paths in $E$. If a term $k_{i}\alpha_{i}\beta_{i}^{\ast}%
=k_{i}v\alpha_{i}\beta_{i}^{\ast}u\neq0$, then $v=s(\alpha_{i})$,
$u=s(\beta_{i})$ and $r(\alpha_{i})=r(\beta_{i})=w$, so $w\in T(u)\cap T(v)$.
Conversely, suppose $w\in T(u)\cap T(v)$. Since $u$ is a line point, so is $w$
and there is a unique path $\mu$ from $u$ to $w$. Then $ua\longmapsto
w\mu^{\ast}ua$ is an isomorphism from the simple module $uL_{K}(E)$ to
$wL_{K}(E)$ with the map $wb\longmapsto u\mu wb$ being the inverse
isomorphism. By a similar argument, $vL_{K}(E)\cong wL_{K}(E)$. Consequently,
$uL_{K}(E)\cong vL_{K}(E)$.
\end{proof}

The next Theorem describes the graphical conditions on $E$ under which
$L_{K}(E)$ is of \ finite irreducible representation type.

\begin{theorem}
\label{Graph Conditions}Let $E$ be an arbitrary graph and let $K$ be any
field. Then the Leavitt path algebra $L_{K}(E)$ is of finite irreducible
representation type if and only if all of the following conditions hold:

(i) $\ \ E$ is acyclic;

(ii) $\ E^{0}$ has only finitely many distinct hereditary saturated \ subsets
$H$ and for each such $H$ the corresponding set $B_{H}$ of breaking vertices
is finite; and

(iii) In the poset of admissible pairs in $E$, $(E^{0},\emptyset)$ is the supremum
of a finite ascending chain
\begin{align*}
(H_{0}  &  =\emptyset,\emptyset)<(H_{1},\emptyset)<(H_{1},B_{H_{1}})<(H_{2},B_{H_{1}}\cap
B_{H_{2}})<(H_{2},B_{H_{2}})<\\
(H_{3},B_{H_{2}}\cap B_{H_{3}})  &  <(H_{3},B_{H_{3}})<\cdot\cdot\cdot
\end{align*}
where, for each $j\geq0$, $H_{j+1}$ is a
hereditary saturated subset of $E^{0}$ and $H_{j+1}\backslash H_{j}$ is the
hereditary saturated closure of the set of line points in $E\backslash
(H_{j},B_{H_{j}})$.
\end{theorem}

\begin{proof}
If $L_{K}(E)$ has finite irreducible representation type, then in the proof of
Theorem \ref{FIRT} it was shown that $E$ is acyclic and that $L_{K}(E)$ has at
most finitely many two-sided ideals (which are all graded). The latter
property is equivalent to Condition (ii) by (\cite{T}, Theorem 5.7). To prove
condition (iii), we shall use the fact, established in Theorem \ref{FIRT},
that $L_{K}(E)$ (and each of its homomorphic images) is semi-artinian. We wish
to construct a chain of ideals of the form $I_{(H,S)}$. Let $J_{1}%
=Soc(L_{K}(E))$. By \cite{AMMS}, $J_{1}$ is the ideal generated by the
hereditary saturated closure $H_{1}$ of the set $T_{1}$ of all the line points
in $E$. Now $J_{1}=I_{(H_{1},\emptyset)}$ and,
by \cite{T}, $J_{1}$ is the kernel of an epimorphism $f:L_{K}%
(E)\longrightarrow L_{K}(E\backslash(H_{1},\emptyset))$ where $(E\backslash
(H_{1},\emptyset))^{0}=E^{0}\backslash H_{1}\cup\{v^{\prime}:v\in B_{H_{1}}\}$ and
that $f(v^{H_{1}})=v^{\prime}$ for all $v\in B_{H_{1}}$. Let $J_{2}%
=I_{(H_{1},B_{H_{1}})}=<H_{1},\{v^{H_{1}}:v\in B_{H_{1}}\}>$. so that
$J_{2}/J_{1}\cong<\{v^{\prime}:v\in B_{H_{1}}\}>\subset L_{K}(E\backslash
(H_{1},\emptyset))$. By \cite{T}, $L_{K}(E)/J_{2}\cong L_{K}(E\backslash
(H_{1},B_{H_{1}}))\cong L_{K}(E\backslash H_{1})$ since $(E\backslash
(H_{1},B_{H_{1}}))^{0}=E^{0}\backslash H_{1}$ and $(E\backslash(H_{1}%
,B_{H_{1}}))^{1}=\{e\in E^{1}:r(e)\notin H_{1}\}$. Now $L_{K}(E)/J_{2}$ has a
non-zero socle (being semi-artinian) and let $H_{2}$ be a hereditary saturated
subset of $E^{0}$ containing $H_{1}$ such that $H_{2}\backslash H_{1}$ is the
hereditary saturated closure of the set $T_{2}$ of all the line points in
$E\backslash(H_{1},B_{H_{1}})$. Define
\[
J_{3}=<J_{2,}H_{2}>=<H_{2},\{v^{H_{1}}:v\in B_{H_{1}}\}>
\]
so that
\[
J_{3}/J_{2}\cong<H_{2}\backslash H_{1}>=Soc(E\backslash(H_{1},B_{H_{1}%
}))\subset L_{K}(E\backslash(H_{1},B_{H_{1}}))\cong L_{K}(E\backslash H_{1}).
\]
If $v\in B_{H_{1}}\backslash B_{H_{2}}$, then $r(s^{-1}(v))\subset
H_{2}\subset J_{3}$ and since $v^{H_{1}}\in J_{3}$, we conclude that $v\in
J_{3}$. In the isomorphism $L_{K}(E)/J_{2}\cong L_{K}(E\backslash H_{1}))$,
$v$ gets mapped to an element in $(J_{3}/J_{2})\cap(E^{0}\backslash
H_{1})=H_{2}\backslash H_{1}$ and so $v\in H_{2}$. Thus $B_{H_{1}}\backslash
B_{H_{2}}\subset H_{2}$ and so $J_{3}=<H_{2},\{v^{H_{2}}:v\in B_{H_{1}}\cap
B_{H_{2}}\}>=I_{(H_{2},B_{H_{1}}\cap B_{H_{2}})}$. Note that, in the
isomorphism $L_{K}(E)/J_{1}\cong L_{K}(E\backslash(H_{1},\emptyset))$, $J_{3}%
/J_{1}$ maps to $Soc(L_{K}(E\backslash(H_{1},\emptyset)))$. Define $J_{4}%
=I_{(H_{2},B_{H_{2}})}$. Proceeding like this, we obtain a chain of ideals
\begin{align*}
\{0\}  &  =I_{(H_{0}=\emptyset,\emptyset)}\subset I_{(H_{1},\emptyset)}\subset I_{(H_{1}%
,B_{H_{1}})}\subset I_{(H_{2},B_{H_{1}}\cap B_{H_{2}})}\subset\\
I_{(H_{2},B_{H_{2}})}  &  \subset I_{(H_{3},B_{H_{2}}\cap B_{H_{3}})}\subset
I_{(H_{3},B_{H_{3}})}\subset\cdot\cdot\cdot\qquad\qquad\qquad\qquad(\ast)
\end{align*}
whose union is $L_{K}(E)$. Here, for each $j\geq0$, $H_{j+1}$ is a hereditary
saturated subset of $E^{0}$ and $H_{j+1}\backslash H_{j}$ is the hereditary
saturated closure of the set $T_{j+1}$ of all the line points in
$E\backslash(H_{j},B_{H_{j}})$. Note that, for each $j\geq0$, there are only
finitely many equivalence classes of line points in $E\backslash
(H_{j},B_{H_{j}})$, due to Condition (ii) . \ Our construction shows that, in
the poset of admissible pairs, $(E^{0},\emptyset)$ is then the supremum of a finite
ascending chain
\begin{align*}
(\emptyset,\emptyset)  &  <(H_{1},\emptyset)<(H_{1},B_{H_{1}})<(H_{2},B_{H_{1}}\cap B_{H_{2}%
})<\\
(H_{2},B_{H_{2}})  &  <(H_{3},B_{H_{2}}\cap B_{H_{3}})<(H_{3},B_{H_{3}}%
)<\cdot\cdot\cdot\qquad\qquad\qquad\qquad(\ast\ast)
\end{align*}
where the sets $H_{j}$ are as described above.

Conversely, Condition (i) implies, by \cite{AR1}, that $L_{K}(E)$ is von
Neumann regular and Condition (ii) implies \ that it has only finitely many
two-sided ideals (which are all graded ideals). Consider chain $(\ast\ast)$
indicated in Condition (iii) and the corresponding chain of ideals $(\ast)$
constructed above. Denote $I_{(H_{1},\emptyset)}$ by $S_{1}$ and, for each $i\geq
2$, denote $I_{(H_{i},B_{H_{i-1}}\cap B_{H_{i}})}$ by $S_{i}$. Then we get the
ascending finite chain
\[
\{0\}\subset S_{1}\subset\cdot\cdot\cdot\subset S_{t}=L_{K}(E)
\]
where $S_{1}=Soc(L_{K}(E))$ and for each $i$, $S_{i+1}/S_{i}=Soc(L_{K(E)}%
/S_{i})$. Hence $L_{K}(E)$ is a semi-artinian ring. By Theorem
\ref{FIRT}, $L_{K}(E)$ is of finite irreducible representation type.
\end{proof}

\begin{example}
To illustrate Theorem \ref{Graph Conditions}, consider the graph
$P_{3}$ in the next section: It is clearly acyclic. Its hereditary saturated
subsets of vertices are: $H_{0}=$ $\emptyset$, the empty set; $H_{1}=\{v_{11}%
,v_{12},v_{13},\cdot\cdot\cdot\}$; $H_{2}=H_{1}\cup\{v_{21},v_{22}%
,v_{23},\cdot\cdot\cdot\}$; $H_{3}=(P_{3})^{0}$ . Since the graph is row-finite there are no breaking vertices.
Then the chain in condition (iii) is given by
\[
(\emptyset,\emptyset)<(H_{1},\emptyset)< (H_{2},\emptyset)< (H_{3}=(P_{3})^{0},\emptyset).
\] 
 \end{example}

\medskip

We point out that one can build examples showing that no two of the three conditions in
Theorem \ref{Graph Conditions} imply the third.

\section{Examples}

In this section, we wish to illustrate Theorems  \ref{CIRT} and \ref{FIRT} by constructing
non-trivial examples of $L_{K}(E)$ which, being semi-artinian (von Neumann
regular) rings, are the union of a finite or infinite  socular
chain
\[
0<S_{1}< S_2 \cdots <S_{n} < \cdots 
\]
where $S_{1}=Soc(L_{K}(E))$ and, for each $j<n$, $S_{j+1}/S_{j}=Soc(L_{K}%
(E)/S_{j})$. To have a really simple example, one can assume that $S_{1}$ and
each $S_{j+1}/S_{j}$ are direct sums of isomorphic simple modules. Since the socle of a Leavitt path algebra is generated by line
points in the graph, the assumed property on $S_{1}$ and $S_{j+1}/S_{j}$ imply
that the line points belonging to $S_{1}$ (and likewise in $S_{j+1}/S_{j}$)
all must form a single straight line segment.

It is this idea that was used to construct the "Pyramid" examples in reference
\cite{ARS}. In some sense these examples are the simplest  
examples of semi-artinian Leavitt path algebras of infinite graphs with arbitrary Loewy length.

We will show that, for each (finite or infinite) cardinal $\kappa$,
the Leavitt path algebra $L_{K}(P_{\kappa})$ of the Pyramid graph $P_{\kappa}$ has exactly $\kappa$ distinct
isomorphism classes of simple right $L_{K}(E)$-modules.

Let $P_{1}$ be the graph 
$$\xymatrix{  {\bullet}^{v_{1,1}} \ar[r]  & {\bullet}^{v_{1,2}} \ar[r]
& {\bullet}^{v_{1,3}} \ar[r]  &
               {\bullet}^{v_{1,4}} \ar@{.>}[r] &  \\
            }$$
consisting of a single infinite path. Now all the vertices $v_{1,i}$ are line
points in $P_{1}$ and so $v_{1,i}L_{K}(P_{1})$ is a simple module for all $i$
(see \cite{AMMS}). Also $L_{K}(P_{1})=\oplus_{i}v_{1,i}L_{K}(P_{1}%
)=Soc(L_{K}(P_{1}))$. From Lemma \ref{Line point} it is clear that, for all
$i<j$, $v_{1,i}L_{K}(P_{1})\cong v_{1,j}L_{K}(P_{1})$. Since $L_{K}(P_{1})$ is a
direct sum of simple modules, every simple right $L_{K}(P_{1})$-module is
isomorphic to the simple right ideal $v_{1,i}L_{K}(P_{1})$ and we conclude that
all the simple right $L_{K}(P_{1})$-modules are isomorphic.

Let $P_{2}$ be the graph
$$
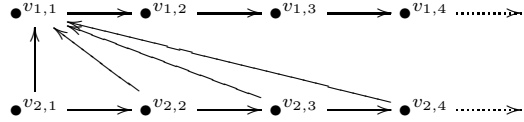
\xymatrix{{\bullet}^{v_{1,1}} \ar[r]  & {\bullet}^{v_{1,2}} \ar[r]
& {\bullet}^{v_{1,3}} \ar[r]  &
               {\bullet}^{v_{1,4}} \ar@{.>}[r] &  \\
              {\bullet}^{v_{2,1}} \ar[r] \ar[u] & {\bullet}^{v_{2,2}}
\ar[r] \ar[ul] & {\bullet}^{v_{2,3}} \ar[r] \ar[ull] &
{\bullet}^{v_{2,4}}\ar@{.>}[r] \ar[ulll] &  \\
               }$$

Now the line points in the graph $P_{2}$ are
the vertices in the first row, namely, $v_{1,1},v_{1,2},v_{1,3},\cdot\cdot\cdot$
\ and they generate the socle $S$ of $P_{2}$ which, from the explanation in
describing $L_{K}(P_{1})$ above, is a direct sum of isomorphic faithful simple
right $L_{K}(P_{2})$-modules. Also, by \cite{T}, $P_{2}/S\cong L_{K}(F)$ where
$F\cong P_{1}$ and so $P_{2}/S=Soc(P_{2}/S)$ is a direct sum of isomorphic
simple modules annihilated by the ideal $S$. Thus $L_{K}(P_{2})$ has exactly
two distinct isomorphism classes of simple $L_{K}(P_{2})$-modules.

Let $P_{3}$ be the graph
$$
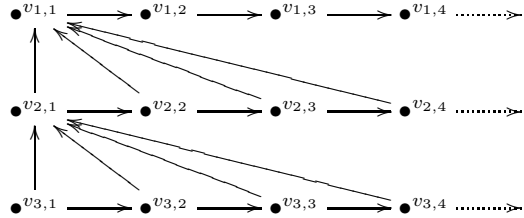
\xymatrix{  {\bullet}^{v_{1,1}} \ar[r]  & {\bullet}^{v_{1,2}} \ar[r]
& {\bullet}^{v_{1,3}} \ar[r]  &
               {\bullet}^{v_{1,4}} \ar@{.>}[r] &  \\
             {\bullet}^{v_{2,1}} \ar[r] \ar[u] & {\bullet}^{v_{2,2}}
\ar[r] \ar[ul] & {\bullet}^{v_{2,3}} \ar[r] \ar[ull] &
{\bullet}^{v_{2,4}}\ar@{.>}[r] \ar[ulll] &  \\
             {\bullet}^{v_{3,1}} \ar[r] \ar[u] & {\bullet}^{v_{3,2}}
\ar[r] \ar[ul] & {\bullet}^{v_{3,3}} \ar[r] \ar[ull] &
{\bullet}^{v_{3,4}}\ar@{.>}[r] \ar[ulll] &  \\
            }$$

As above, the vertices
$v_{1,1},v_{1,2},v_{1,3},\cdot\cdot\cdot$ generate the socle $S$ of $L_{K}%
(P_{3})$ which is a direct sum of isomorphic faithful simple $L_{K}(P_{3}%
)$-modules and $L_{K}(P_{3})/S\cong L_{K}(P_{2})$. Clearly, by Lemma
\ref{Quotient simple} and the description of $L_{K}(P_{2})$ above,
$S_{2}/S\cong Soc(L_{K}(P_{2}))$ is a direct sum of isomorphic simple
$L_{K}(P_{3})$-modules annihilated by the ideal $S$ and that 
$$L_{K}
(P_{3})/S_{2}\cong L_{K}(P_{2})/Soc(L_{K}(P_{2}))\cong Soc[L_{K}
(P_{2})/Soc(L_{K}(P_{2})]$$ 
is a direct sum of isomorphic simple $L_{K}(P_{3})$-modules annihilated by the ideal $S_{2}$. Thus we conclude that
$L_{K}(P_{3})$ has exactly three distinct isomorphism classes of simple right
$L_{K}(P_{3})$-modules.

Proceeding like this, we conclude, by simple induction, that for any
positive integer $n$, the Leavitt path algebra $L_{K}(P_{n})$ of the ``Pyramid''
graph $P_{n}$ with $n$ ``layers" has exactly $n$ distinct isomorphism classes
of simple right $L_{K}(P_{n})$-modules.

Let $P_{\omega}=
{\textstyle\bigcup\limits_{n\in\mathbb{N}}}
P_{n}$ be the ``Pyramid" graph of length $\omega$ constructed inductively and
represented pictorially as follows.

$$
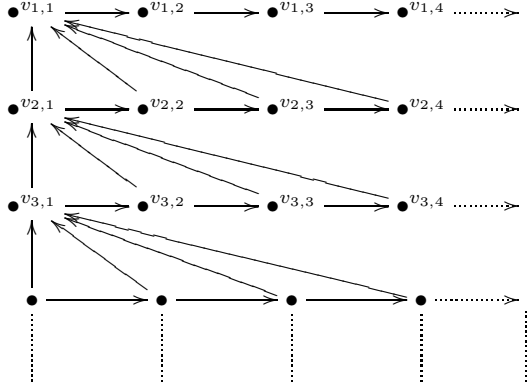
\xymatrix{ {\bullet}^{v_{1,1}} \ar[r]  & {\bullet}^{v_{1,2}} \ar[r]
& {\bullet}^{v_{1,3}} \ar[r]  &
               {\bullet}^{v_{1,4}} \ar@{.>}[r] &  \\
              {\bullet}^{v_{2,1}} \ar[r] \ar[u] & {\bullet}^{v_{2,2}}
\ar[r] \ar[ul] & {\bullet}^{v_{2,3}} \ar[r] \ar[ull] &
{\bullet}^{v_{2,4}}\ar@{.>}[r] \ar[ulll] &  \\
              {\bullet}^{v_{3,1}} \ar[r] \ar[u] & {\bullet}^{v_{3,2}}
\ar[r] \ar[ul] & {\bullet}^{v_{3,3}} \ar[r] \ar[ull] &
{\bullet}^{v_{3,4}}\ar@{.>}[r] \ar[ulll] &  \\
{\bullet} \ar[r] \ar[u] \ar@{.}[d] & {\bullet}
\ar[r] \ar[ul] & {\bullet} \ar[r] \ar[ull] &
{\bullet} \ar@{.>}[r] \ar[ulll] &  \\
 \ar@{.}[u] & \ar@{.}[u] & \ar@{.}[u] & \ar@{.}[u] & \ar@{.}[u] & 
    \\            }$$

Again, by induction, it follows that $L_{K}(P_{\omega})$ has exactly $\omega$
distinct isomorphism classes of simple $L_{K}(P_{\omega})$-modules.

The graph $P_{\omega+1}$ is obtained from the graph $P_{\omega}$ by adding a
single vertex $v_{\omega+1}$ and connecting it by an edge to each of the
vertices $v_{j,1}$ for $j<\omega$ in the graph $P_{\omega}$. Specifically,
$(P_{\omega+1})^{0}=(P_{\omega})^{0}\cup\{v_{\omega+1}\}$, $(P_{\omega+1}%
)^{1}=(P_{\omega})^{1}\cup\{e_{\omega+1,j}:j<\omega\}$ where, for each $j$,
$s(e_{\omega+1,j})=v_{\omega+1}$ and $r(e_{\omega+1,j})=v_{j,1}$. If
$S_{\omega}$ denotes the $\omega$-socle being the ideal generated by all the
vertices in $P_{\omega}$, then $L_{K}(P_{\omega+1})/S_{\omega}$ is a simple
$L_{K}(P_{\omega+1})$-module whose annihilator ideal is $S_{\omega}$ and it is
isomorphic to the Leavitt path algebra of a graph $\{\overset{v_{\omega+1}%
}{\bullet}\}$ consisting of a single vertex and no edges. Clearly
$L_{K}(P_{\omega+1})$ has $\omega+1$ distinct isomorphism classes of simple
$L_{K}(P_{\omega+1})$-modules.

Proceeding this way, as was shown in \cite{ARS}, we can construct, by
transfinite induction, a ``Pyramid'' graph $P_{\lambda}$ for each ordinal
$\lambda$. The Leavitt path algebra $L_{K}(P_{\lambda})$ is a semi-artinian
von Neumann ring of Loewy length $\lambda$ such that, for each $\alpha
<\lambda$, the quotient $S_{\alpha+1}/S_{\alpha}$ of successive socles is
isomorphic to the Leavitt path algebra of an infinite line segment (like the
graph $P_{1}$) or a graph $\{\overset{v}{\bullet}\}$ consisting of a single
vertex and no edges, according as $\alpha$ is a successor or a limit ordinal.
Thus $S_{\alpha+1}/S_{\alpha}$ has exactly one isomorphism class of simple
modules (annihilated by $S_{\alpha}$). By transfinite induction, one can then
show that the Leavitt path algebra $L_{K}(P_{\lambda})$ has exactly
$|\lambda|$ isomorphism classes of simple $L_{K}(P_{\lambda})$-modules.

\bigskip

\begin{acknowledgement}
1. We are deeply grateful to the referee whose insightful comments, questions
and suggestions lead to a substantial improvement of the previous version of
this paper.

2. The first-named author was supported by DGI MICIIN-FEDER
MTM2011-28992-C02-01, and by the Comissionat per Universitats i Recerca de la
Generalitat de Catalunya.

3. Part of this work was done when the second-named author visited the
Universitat Autonoma de Barcelona during May 2013 and he gratefully
acknowledges the support and the hospitality of the faculty of the Department
of Mathematics. 

\end{acknowledgement}

\end{document}